\documentclass[12pt]{amsart}
\usepackage[colorlinks=true,citecolor=green,linkcolor=magenta]{hyperref}
\usepackage{amsmath}
\usepackage{verbatim}
\usepackage{amsfonts}
\usepackage{amssymb}
\usepackage{blkarray}
\usepackage{color,colortbl}
\usepackage[all]{xy}  
\usepackage{enumerate}
\usepackage[top=1in, bottom=1in, left=1in, right=1in]{geometry}
\usepackage{mathrsfs}

\usepackage{stmaryrd}
\usepackage{bbold}
\renewcommand{\k}{\mathbb{k}}
\usepackage{cleveref}
\usepackage{soul}
\usepackage{arydshln}
\usepackage{tikz-cd}
\usepackage{graphicx}
\theoremstyle{definition}
\newtheorem{theorem}{Theorem}[section]
\newtheorem{theoremx}{Theorem}

\numberwithin{equation}{section}

\newtheorem{corollary}[theorem]{Corollary}
\newtheorem{lemma}[theorem]{Lemma}
\newtheorem{proposition}[theorem]{Proposition}

\newtheorem{notation}[theorem]{Notation}

\theoremstyle{definition}
\newtheorem{definition}[theorem]{Definition}

\newtheorem{example}[theorem]{Example}

\newtheorem{conjecturex}[theoremx]{Conjecture}

\newtheorem{remark}[theorem]{Remark}

\newtheoremstyle{TheoremNum}
        {8pt}{8pt}              
        {\upshape}                      
        {}                              
        {\bfseries}                     
        {.}                             
        {.5em}                             
        {\theoremname{#1}\theoremnote{ \bfseries #3}}
  \theoremstyle{TheoremNum}

\newcommand{\m}{\mathfrak{m}}
\newcommand{\n}{\mathfrak{n}}

\renewcommand{\(}{\left(}
\renewcommand{\)}{\right)}

\newcommand{\cC}{\mathfrak{C}}

\newcommand{\ord}{\operatorname{ord}}

\newcommand{\edim}{\operatorname{edim}}
\newcommand{\Rank}{\operatorname{rank}}
\newcommand{\type}{\operatorname{type}}

\newcommand{\Hom}{\operatorname{Hom}}

\newcommand{\Ext}{\operatorname{Ext}}

\newcommand{\C}{\mathfrak{C}}

\newcommand{\Tor}{\operatorname{Tor}}
\newcommand{\quot}{\operatorname{Quot}}

\newcommand{\Char}{\operatorname{char}}

\newcommand{\rk}{\operatorname{rk}}

\renewcommand{\leq}{\leqslant}
\renewcommand{\geq}{\geqslant}
\newcommand{\ds}{\displaystyle}

\newcommand{\ps}[1]{\llbracket {#1} \rrbracket}

\newcommand\scalemath[2]{\scalebox{#1}{\mbox{\ensuremath{\displaystyle #2}}}}

\title{Torsion in differentials and Berger's Conjecture}

\address{Department of Mathematics, University of Virginia, Charlottesville, VA 22904-4135, USA}
\author[Huneke]{Craig Huneke}
\email{huneke@virginia.edu}
\author[Maitra]{Sarasij Maitra}
\email{sm3vg@virginia.edu}
\address{Department of Mathematics, Indian Institute of Technology Delhi, India}
\author[Mukundan]{Vivek Mukundan}
\email{vmukunda@iitd.ac.in}


\dedicatory{This paper is dedicated to J\"urgen Herzog, whose  fundamental\\ research in commutative algebra has inspired researchers for  50 years.}

\begin{document}
\maketitle
\begin{abstract}
Let $(R,\m,\k)$ be an equicharacteristic one-dimensional complete local domain over an algebraically closed field $\k$ of characteristic $0$. R. Berger conjectured that $R$ is regular if and only if the universally finite module of differentials $\Omega_R$ is a torsion-free $R$ module. We give new cases of this conjecture by extending works of G\"uttes (\cite{Guttes1990}) and Corti\~{n}as, Geller  and Weibel (\cite{ABC1}). This is obtained by constructing a new subring $S$  of $\Hom_R(\m,\m)$ and constructing enough torsion in $\Omega_S$, enabling us to pull back a nontrivial torsion to $\Omega_R$.
\end{abstract}

\section{Introduction} This paper gives new cases of a conjecture made by R. Berger in 1963 \cite{Berger63}. Let $\k$ be an algebraically closed field of characteristic $0$, and let
$(R,\mathfrak{m}_R,\k)$ be  an equicharacteristic reduced one-dimensional complete local $\k$-algebra.  Berger conjectured that the universally finite module of differentials, $\Omega_R$,
is torsion-free if and only if $R$ is regular.  The case in which $R$ is regular is easy, since in that case $\Omega_R$ is free. Hence another formulation of
his conjecture is that $\Omega_R$ is torsion-free if and only if it is a free $R$-module.

There are many approaches to the conjecture which have been partially successful. We refer  \cite{Berger63}, \cite{scheja1970differentialmoduln}, \cite{herzog1971wertehalbgruppe}, \cite{Bassein}, \cite{Herzog78}, \cite{Buchweitz}, \cite{Ulrich81}, \cite{Koch}, \cite{HerzogWaldi84},  \cite{HerzogWaldi86}, \cite{yoshino1986torsion}, \cite{Berger88}, \cite{pohl1989torsion1}, \cite{Guttes1990}, \cite{Hubl}, \cite{Isogawa}, \cite{Pohl91}, \cite{ABC1}, \cite{ABC2} and \cite{maitra2020partial} for these approaches. A very nice summary of a majority of these results, along with the main ideas of proofs, can be found in \cite{Berger_article}.

Our generalizations have to do with how the conductor $\C_R=R:_{\quot(R)}\overline{R}$ of $R$ sits inside $R$.  We first prove that if the conductor is not in the square of the maximal ideal, then Berger's conjecture is true (\Cref{valuation more than conductor}). When the conductor is in the square of the maximal ideal, we construct a certain subring $S$ of $\Hom_R(\m_R,\m_R)$. By construction, there always exists torsion in $\Omega_S$. We show that if there are enough torsion elements in $\Omega_S$, we can construct a nonzero torsion element in $\Omega_R$ (\Cref{thm on one more torsion}). One of the first cases we prove is if $S$ is quasi-homogeneous then
Berger's conjecture is true (\Cref{quasihomogeneous}), generalizing a result of Scheja \cite{scheja1970differentialmoduln}. 

Let $x$ be a minimal reduction of the maximal ideal $\m_R$. The next set of results depends on which power of the maximal ideal $\m_R$ is contained in the conductor $\C_R$ of $R$. We study the quantity $s(R):=\dim_\k\frac{\(\C_R,x\)}{(x)}$ which we shall refer to as the \textit{reduced type} of $R$. The terminology reduced type is natural due to the fact that $\frac{(\C_R,x)}{(x)}\subseteq \frac{(x):\m}{(x)}$ and the $\k$-dimension of the latter module is precisely the type of $R$.

Our main results extend both those of G\"uttes (\cite{Guttes1990}), who proved that if either $\m_R^4\subseteq xR$ or $R$ is Gorenstein and $\m_R^5\subseteq xR$, then Berger's conjecture holds, and of Corti\~{n}as, Geller, and Weibel (\cite{ABC1}), who
proved that if $\m_R^3\subseteq \C_R$, then Berger's conjecture holds. 

	
 We summarize our extensions below, with $\edim$ denoting the embedding dimension (\Cref{m^4 in C} and \Cref{m^6 in C}):

\begin{theoremx}
Let $(R,\mathfrak{m}_R,\k)$ be  an equicharacteristic one-dimensional complete local domain over an algebraically closed field $\k$ of characteristic $0$,
$x$ be a minimal reduction of $\mathfrak{m}_R$ and $\C_R$ be the conductor ideal of $R$ in its integral closure. Further, let $n=\edim R$ and $s(R)$ be the reduced type of $R$. Then Berger's conjecture is true in the following cases:
\begin{enumerate}
\item $\m_R^4\subseteq (\C_R, x)$ and $2\cdot s(R)\leq n(n- 3)$,
\item $\m_R^6\subseteq (x)$, $n\geq 6$ and $R$ is Gorenstein.
\end{enumerate}
\end{theoremx}

We give other generalizations which relate to the structure of the ring $S$. In particular, when the length of $S/R$ is one, we obtain results similar in spirit to the case in which $R$ is Gorenstein, without having to assume the Gorenstein property. 

The structure of the article is as follows: \Cref{conductor not in m^2} takes care of the case when the conductor $\C_R$ is not contained in the square of the maximal ideal $\m_R$ (\Cref{valuation more than conductor}) and some related cases. \Cref{basics of S} gives the details of the construction of $S$ and its basic properties (\Cref{conductorbasicprops}, \Cref{canonical module of S} and \Cref{kernelofnew_proof}). And finally, the main results are presented in \Cref{main results} (\Cref{m^4 in C}, \Cref{mainthm} and \Cref{m^6 in C}).

\medskip

\section{Setting and Preliminaries}\label{prelims}
Throughout this paper we assume that $\k$ is an algebraically closed field of characteristic $0$, and $(R,\mathfrak{m}_R,\k)$ is a one-dimensional complete equicharacteristic local $\k$-algebra which is a domain with embedding dimension (denoted by $\edim$) $n$, i.e. $\mu_R(\mathfrak{m}_R)=n$ where $\mu_R(M)$ denotes the minimal number of generators for any $R$-module $M$. We also denote the length of any $R$-module $M$ by $\ell(M)$. 

Choosing $t$ to be a uniformizing parameter of the integral closure $\overline{R}$ of $R$,  we may assume that
$\overline{R}=\k\ps{t}$. It follows we can write $R=\k\ps{\alpha_1 t^{a_1},\dots,\alpha_nt^{a_n}}$ where $\alpha_i$'s are units in $\overline{R}$ and $a_1\leq a_2\leq \cdots\leq a_n$.  Note also that here $\ds \bar{R}$ is finitely generated over $R$ (see for example \cite{SwansonHuneke}[Theorem 4.3.4]).

We define an epimorphism 
\begin{align}
\Phi:P=\k\ps{X_1,\dots,X_n}&\twoheadrightarrow R\\
\Phi(X_i)&=\alpha_it^{a_i}\text{ for } 1\leq i\leq n.\nonumber 
\end{align}
We denote the kernel of $\Phi$ by $I=(f_1,\dots,f_m)$ and hence have the natural isomorphism $R\cong \k\ps{X_1,\dots,X_n}/I$. Since $\edim R=n$, $I$ is contained in $\m^2$ where $\m=(X_1,\dots,X_n)$, the maximal ideal of $\k\ps{X_1,\dots,X_n}$.  Such rings are called \textit{analytic $\k$-algebras}. We will interchangeably use $\alpha_i t^{a_i}$ for $x_i$, the images of $X_i$ in the quotient $\k\ps{X_1,\dots,X_n }/I$. 

We also define the valuation $\ord(-)$ on $\overline{R}$ given by $\ord(p(t))=a$ if $p(t)=t^a\alpha$ where $\alpha$ is a unit in $\k\ps{t}$  (see for example \cite[Example 6.7.5]{SwansonHuneke}).

\subsection{Universally Finite Module of Differentials}

\begin{definition}\label{defmoduleofdiff}
	Let $R$ be an analytic one-dimensional $\k$-algebra as above, which is a domain. Let $I=(f_1,\dots,f_m)$ where $f_j\in P=\k\ps{X_1,\dots,X_n}$. We assume that $I\subseteq \m_P^2$ where $\m_P=(X_1,\dots,X_n)$. Then the \textit{universally finite module of differentials over $\k$}, denoted by $\Omega_{R}$, has a (minimal) presentation given as follows:
	$$R^{m}\xrightarrow{\left[\frac{\partial f_j}{\partial x_i}\right]} R^n\to \Omega_{ R}\to 0$$ where $\left[\frac{\partial f_j}{\partial x_i}\right]$ is the Jacobian matrix of $I$, with entries in $R$. 
\end{definition}
We refer the reader to the excellent resource \cite{Kunzbook} for more information.

Let $\ds \tau(\Omega_{ R})$ denote the torsion submodule of $\Omega_{ R}$. The conjecture of interest in this article is the following: 
\begin{conjecturex}[R. W. Berger \cite{Berger63}]
	Let $\k$ be a perfect field and let $R$ be a reduced one-dimensional analytic $\k$-algebra. Then $R$ is regular if and only if $\ds \tau(\Omega_{ R})=0$. 
\end{conjecturex}

Although the conjecture is for reduced algebras, in this paper we only deal with the case in which $R$ is a domain and $\k$ is algebraically closed of characteristic zero. Our techniques do not immediately seem
to apply otherwise.

\begin{remark}\label{rankomega}
	When $\k$ is a perfect field, it is well-known that
	$\Rank_R(\Omega_{R})=\dim(R)=1$. Hence from \Cref{defmoduleofdiff}, we get that $\Rank A=n-1$ where $A=\left[\frac{\partial f_j}{\partial x_i}\right]$ is the Jacobian matrix of $I$, with entries in $R$.
\end{remark}
\begin{remark}
	It is clear from \Cref{defmoduleofdiff}, that $\tau(\Omega_{ R})=0$ when $R$ is regular. Thus, from now on we assume that $n\geq 2$. 
\end{remark}

\subsection{The Conductor}\label{conductor def}
The conductor ideal $\cC_R$ will be crucial for the purposes of this paper. Recall that the conductor is the largest common ideal of $R$ and its
integral closure, $\overline R $.  It follows that $\ds \cC_R=R:_Q\overline{R}$ where $Q=\quot(R)$, denotes the fraction field of $R$. Since $\ds \overline{R}=k\ps{t}$ and $\cC_R$ is an ideal of $\bar{R}$ as well, we have that $\cC_R=(t^i)_{i\geq c_R}$ where $c_R$ is the least integer such that $t^{c_R-1}\not \in R$, and $t^{c_R+i}\in R$ for all $i\geq 0$. The number $c_R$ is characterized as the least valuation in $\cC_R$. It is clear from this  discussion that there cannot be any element $r\in R$, such that $v(r)=c_R-1$.  Since $\ds \bar{R}$ is finitely generated over $R$ (\cite[Theorem 4.3.4]{SwansonHuneke}),  the conductor ideal is a nonzero ideal of $R$, and it is never all of $R$ unless $R$ is regular.

\subsection{Computing Torsion}
We have the following commutative diagram  using the functorial universal properties of the module of differentials and the associated universal derivations.

\[\begin{tikzcd}
\Omega_{R}\ar[r,"f"]  & \Omega_{\overline{R}}\\
R\ar[r,hook,"i"] \ar[u,"d"] & \overline{R}\ar[u, swap, "d"]
	\end{tikzcd}
\]
We use the same symbols $d$ for both the vertical maps.
Since, $\rk(\Omega_R)=\rk(\Omega_{\overline{R}})=1$ ($R$ and $\bar{R}$ have the same fraction field), and  $\Omega_{\overline{R}}$ is free over $\overline{R}$, we get that the $\tau(\Omega_R)=\ker f$.  

Also note that by commutativity of the diagram, 

$$f(dx_i)=\frac{dx_i}{dt}dt.$$

Since $\Omega_{\overline{R}}$ is isomorphic to $\overline{R}$, we  see that $\Omega_{R}$ surjects to a $R$-submodule $\ds \sum_{i=1}^nR\frac{dx_i}{dt}$ of $\ds \overline{R}=\k\ps{t}$. This is a fractional ideal in $\overline{R}$, so multiplying by a suitably high enough power of $t$, it is isomorphic to an ideal of $R$.

The torsion submodule $\tau(\Omega_R)$  is the kernel of the map $\Omega_R\rightarrow\Omega_{\overline{R}}$. Thus, from the above discussion, we get that $\ds \tau(\Omega_R)$  consists of the tuples $\begin{bmatrix}
r_1\\\vdots\\r_n
\end{bmatrix}$ such that $\ds \sum_{i=1}^nr_i\frac{dx_i}{dt}=0$. Evidently, $\tau(\Omega_R)$ is non-zero precisely when the tuples $\begin{bmatrix}
r_1\\\vdots\\r_n
\end{bmatrix}$ are not in the image of the presentation matrix (Jacobian matrix of $I$) of $\Omega_R$, all entries written in terms of the uniformizing parameter $t$.  This provides one computational way of computing torsion using Macaulay 2. 

\begin{example}
$R=\ps{t^3,t^4,t^5}$ and its defining ideal $I=(y^2-xz,z^2-x^2y,x^3-yz)$ in $\k\ps{x,y,z}$. 
Consider the element $$\tau=4ydx-3xdy=\begin{bmatrix}
4y\\-3x\\0
\end{bmatrix}\hspace{-1em}
\begin{array}{c;{2pt/2pt}c}
~&~\\ ~&~\\ ~&~\\
\end{array}\hspace{-1em}
\overbrace{
	\begin{array}{c}
	dx \\ dy \\dz
	\end{array}
}^{\text{basis}}$$ in $\Omega_R$. Clearly $4y\frac{dx}{dt}-3x\frac{dy}{dt}=4t^4(3t^2)-3t^3(4t^3)=0$. Now the presentation matrix of $\Omega_{R/\m_R^2}$ is
\begin{align*}
\begin{bmatrix}
2x & y & z & 0 & 0 &0\\
0 & x & 0 & 2y &z &0\\
0 & 0 & x & 0 &y & 2z
\end{bmatrix}.
\end{align*}
Since the image of $\tau$ in $\Omega_{R/\m_R^2}$ can never be written as a linear combination of the columns of the above presentation, $\tau$ is nonzero in $\Omega_{R/\m_R^2}$. Thus $\tau$ is nonzero in $\Omega_R$ as well.
\end{example}
\section{Nonzero torsion when $\mathfrak{C}_R\not\subseteq\m_R^2$}\label{conductor not in m^2}
Throughout this section, we assume that $\k$ is an algebraically closed field of characteristic $0$, and $(R,\mathfrak{m}_R,\k)$ is a one-dimensional complete equicharacteristic local $\k$-algebra which is a domain with embedding dimension $n$. Choosing $t$ to be a uniformizing parameter of the integral closure $\overline{R}$,  we may assume that $\overline{R}=\k\ps{t}$. It follows we can write $R=\k\ps{\alpha_1 t^{a_1},\dots,\alpha_nt^{a_n}}$ where $\alpha_i$'s are units in $\overline{R}$ and $a_1\leq a_2\leq \cdots\leq a_n$. Let $\C_R$ denote the conductor ideal.

Our primary construction which appears in the sequel will make use of the condition $\C_R\subseteq \m_R^2$. So prior to our main construction, we settle the case $\C_R\not\subseteq\m_R^2$ by showing that this condition always leads to nonzero torsion in $\Omega_R$. In this case at least one of the minimal generators $x_1,\dots,x_n$ of the maximal ideal is in the conductor $\mathfrak{C}_R$. Thus, after a change of variables, the minimal generators in the conductor can be replaced by monomials (i.e., if $x_i=\alpha_it^{a_i}\in \mathfrak{C}_R$, then after a change of variables, $\alpha_i$ can be chosen to be a unit in $R$). We will see 	 in \Cref{two monomials}, the presence of monomials will lead to nonzero torsion in $\Omega_R$. 

\begin{theorem}\label{valuation more than conductor} If $\mathfrak{C}_R\not\subseteq\m_R^2$, 
	 then the torsion $\tau(\Omega_R)$ is nonzero.
\end{theorem}
\begin{proof} Write  $R= \k\ps{\alpha_1t^{a_1},\dots,\alpha_nt^{a_n}}$ with conductor $\C_R=(t^{c_R})\overline{R}$. We first monomialize the $r^{th}$ term as follows:
by multiplying by a nonzero element of $\k$, we may assume that the constant term of the unit $\alpha_r$ is $1$. 
By Hensel's lemma \cite[Theorem 7.3]{eisenbud_Commalg}, there exists an element $\beta\in R$ such that $\beta^{a_r} = \alpha_r$.  Here we use that the characteristic of $\k$ is $0$.  We
write $\beta = 1 + \beta_1t+\cdots$. 
Consider the change of variables $s=\beta t$. Under this change of variables, notice that $\k\ps{t}=\k\ps{s}$. Now
\begin{align*}
	s=\beta t=t+\beta_1t^{2}+\beta_2t^{3}+\cdots.
\end{align*}
Note that $s^{a_r}=(\beta t)^{a_r}=\alpha_rt^{a_r}\in R$. Furthermore, $\alpha_it^{a_i}=\alpha_i's^{a_i},1\leq i\neq r\leq n$, where $\alpha'_i\in\overline{R}$ are units. Then $R=\k\ps{\alpha_1t^{a_1},\dots,\alpha_nt^{a_n} } = \k\ps{\alpha_1's^{a_1},\dots,\alpha_{r-1}'s^{a_{r-1}},s^{a_r},\alpha_{r+1}'s^{a_{r+1}},\dots ,\alpha_n's^{a_n}}$. We apply this change of variables  with $r = 1$ to assume without loss of generality, for the remainder of this proof, that  $R = \k\ps{t^{a_1},\alpha_2t^{a_2},\cdots,\alpha_nt^{a_n}}$. 

Since $\C_R$ is not contained in $\m_R^2$, we must have that $a_n\geq c_R$. Write $\alpha_n = \alpha_{n0} + tb$, where $\alpha_{n0}\ne 0$ is in $\k$ and $b\in \overline{R}$. 
Then $\alpha_nt^{a_n} = \alpha_{n0}t^{a_n} + t^{a_n+1}b$, where $b\in \overline{R}$. However, since $a_n\geq c_R$, it follows that $t^{a_n+1}b\in \C_R\subset R$.
Hence, $t^{a_n} = \alpha_{n0}^{-1}(\alpha_nt^{a_n}-t^{a_n+1}b)\in R$ as well, and then $R = \k\ps{t^{a_1}, \alpha_2t^{a_2},\ldots,\alpha_{n-1}t^{a_{n-1}},t^{a_n}}$.

We now use this particular form for $R$ to find a nonzero torsion element in $\Omega_R$. Namely, $a_nx_ndx_1-a_1x_1dx_n\in\Omega_R$ and the exact sequence
	\begin{align*}
		0\rightarrow\tau(\Omega_R)\rightarrow\Omega_R\xrightarrow{\phi} R\frac{dx_1}{dt}+\cdots+R\frac{dx_n}{dt}\rightarrow 0 
	\end{align*}
	where the map $\phi$ is the $R$-module map given by $\phi(dx_i)=\frac{dx_i}{dt}, 1\leq i\leq n$. Under this map 
	\begin{align*}
		\phi(a_nx_ndx_1-a_1x_1dx_n)&=a_nt^{a_n}\frac{dt^{a_1}}{dt}-a_1t^{a_1}\frac{dt^{a_n}}{dt}\\
		&=(a_1a_nt^{a_1+a_n-1}-a_na_1t^{a_1+a_n-1})dt=0
	\end{align*}
	Thus $a_nx_ndx_1-a_1x_1dx_n\in \tau(\Omega_R)$. It remains to see that it is nonzero. Consider the image of this element  $\overline{a_nx_ndx_1-a_1x_1dx_n}$ in $\Omega_{R/\m_R^2}$.  As $x_1x_n\in \m_R^2$, it follows that in $\Omega_{R/\m_R^2}$, $\overline{x_1dx_n+ x_ndx_1} = 0$. Hence $\overline{a_nx_ndx_1-a_1x_1dx_n} = (a_1+a_n)\overline{x_1}d(\overline{x_n})$ in $\Omega_{R/\m_R^2}$. 
	Now using \cite[Proposition 2.6, Corollary 2.7]{ABC1}, we have $(a_1+a_n)\overline{x_1}d(\overline{x_n})\neq 0$ in $\Omega_{R/\m_R^2}$. Thus $a_nx_ndx_1-a_1x_1dx_n\neq 0$ in $\Omega_R$.
	\end{proof}

\begin{example}\label{example with a_i is more than c}
	Let $R=\k\ps{t^4+t^5,t^7+t^{10},t^8+t^{10},t^9+t^{10}}$ . Macaulay2 computations show that the conductor is $\C_R=(t^{c_R})\overline{R}=(t^7)\overline{R}$ . Since $a_2\geq c_R$, we have $\tau(\Omega_R)\neq 0$ using the previous result. 
\end{example}

\begin{remark}\label{two monomials}
	If $\alpha_i,\alpha_j$ are units in $R$ for some $i\neq j$, then we can show that  $\tau(\Omega_R)$ is nonzero. Assuming $i=1,j=2$ we can easily see that $R=\k\ps{x_1,\dots,x_n}\cong \k\ps{t^{a_1},t^{a_2},\alpha_3t^{a_3},\dots,\alpha_nt^{a_n}}$. Notice that $a_2x_2dx_1-a_1x_1dx_2\in \tau(\Omega_R)$. This torsion element is nonzero as $\overline{a_2x_2dx_1-a_1x_1dx_2}=(a_2-a_1)\overline{x_2dx_1}$ is nonzero in $\Omega_{R/\m_R^2}$ (\cite[Corollary 2.7]{ABC1}).
\end{remark}
The above result is also a generalization of \cite[Corollary 3.7]{Isogawa}. The next example illustrates the remark.
\begin{example}
	Let $R=\k\ps{t^{8}+t^{9},t^{9}+t^{15},t^{12}+t^{20},t^{14}}$, the conductor  $\C_R=(t^{c_R})\overline{R}=(t^{20})\overline{R}$. Thus $R \cong \k\ps{t^{8}+t^{9},t^{9}+t^{15},t^{12},t^{14}}$ and hence $R$ has at least one torsion element using \Cref{two monomials}. Notice that in this case, none of the $a_i$'s are bigger than the $c_R$.
\end{example}

\section{The transform $R[\frac{\mathfrak{C}_R}{x_1}]$}\label{basics of S}
Throughout this section, we again assume that $(R,\mathfrak{m}_R,\k)$ is a one-dimensional complete equicharacteristic local $\k$-algebra which is a domain with embedding dimension $n$. Further, $\k$ is an algebraically closed field of characteristic $0$. Choosing $t$ to be a uniformizing parameter of the integral closure $\overline{R}$,  we may assume that $\overline{R}=\k\ps{t}$. Using the technique as in the first paragraph of the proof of \Cref{valuation more than conductor}, we can also write $R=\k\ps{ t^{a_1},\alpha_2t^{a_2},\dots,\alpha_nt^{a_n}}$ where $\alpha_i$'s are units in $\overline{R}$ and $a_1\leq a_2\leq \cdots\leq a_n$. Let $\C_R$ denote the conductor ideal.

In this section we study the main construction $S=R[\frac{\C_R}{x_1}]$ where $x_i=\alpha_it^{a_i}$ with $\alpha_1=1$. Throughout this section, we assume that $\mathfrak{C}_R\subseteq \m_R^2$.

\subsection{Basics of $R[\frac{\mathfrak{C}_R}{x_1}]$}
\qquad \\

We write $\frac{\mathfrak{C}_R}{x_1}$ to denote the set of elements of the form $\frac{c}{x_1}$ where $c\in \C_R$.  We note that the conductor is never inside a proper principal ideal (follows, for instance, from \cite[Corollary 2.6]{maitra2020partial}),
so there are always elements in $\frac{\mathfrak{C}_R}{x_1}$ which are not in $R$ itself. Recall from the introduction that we define the \textit{reduced type} of $R$ to be $$s(R):=\dim_\k\frac{(\C_R,x_1)}{(x_1)}.$$ We use the notation $s$ whenever the underlying ring is clear. 

\smallskip

\begin{lemma}\label{conductorbasicprops} Let $S=R[\frac{\C_R}{x_1}]$.
	The following statements hold. 
	\begin{enumerate}
		\item $ \frac{\mathfrak{C}_R}{x_1}\subseteq \overline{R}$.
		
		\item $\m_R (\frac{\C_R}{x_1})\subseteq \C_R$. In particular, $S\subset \Hom_R(\m,\m)$.
		
		\item Let $\mathfrak{C}_R\subseteq \m_R^2$. Let $\alpha, \beta\in \frac{\mathfrak{C}_R}{x_1}$.  Then $\alpha\beta\in \mathfrak{C}_R$.
		
		\item Let $\ds \mathfrak{C}_S$ denote the conductor of $S$ in $\overline{R}$. Then $ \mathfrak{C}_S=\frac{\mathfrak{C}_R}{x_1}$.
		
\item $S/R$ is a vector space over $\k$ of dimension $s$, where $s$ is the reduced type of $R$.
	\end{enumerate}
\end{lemma}

\begin{proof} Since $\ds (x_1)$ is a minimal reduction of $\m_R$, we have $\frac{\m_R}{x_1}\subseteq \overline{R}$. Hence $ \frac{\mathfrak{C}_R}{x_1}\subseteq \overline{R}$ proving $(1)$.

For $(2)$ note that $\ds \m_R\overline{R}=x_1\overline{R}$ and hence $\ds \m_R\mathfrak{C}_R=x_1\mathfrak{C}_R$. So, $\m_R\frac{\mathfrak{C}_R}{x_1}\subseteq \mathfrak{C}_R\subseteq \m_R$.

Write $\alpha=\frac{c}{x_1},\beta=\frac{c'}{x_1}$ where $\ds c,c'\in \mathfrak{C}_R$. For $(3)$, first note that $\ds cc'\in \mathfrak{C}_R^2\subseteq \mathfrak{C}_R\m_R^2=\mathfrak{C}_R x_1\m_R =x_1^2\mathfrak{C}_R$ as in the proof of $(2)$. This proves $(3)$.

Finally, note that every valuation more than $c_R-a_1-1$ is present in the valuation semi-group of $S$. But there cannot be any element with valuation $c_R-a_1-1$ in $S$: if possible, let $r$ be such an element. Then by $(2)$, $rx_1\in R$ and has valuation $c_R-1$, a contradiction to the choice of $c_R$. This finishes the proof of $(4)$.

For (5), first note that (2) clearly implies that $S/R$ is a vector space. Now $S/R$ is generated as a $\k$-vector space by elements of the form $c/x_1$ such that $c\in \C_R$. Choose a basis $\overline{c_1},\dots,\overline{c_s}$ for the $\k$-vector space $(\C_R,x_1)/(x_1)$. Now construct a $\k$-linear map $\eta:(\C_R,x_1)/(x_1)\rightarrow S/R$ by mapping $\eta(\overline{c_i})=c_i/x_1$. Suppose $\overline{c}=\sum_{i}k_i\overline{c_i}$ with $k_i\in \k$ be such that $\eta(\overline{c})=0$. Then we get $\sum_ik_i\frac{c_i}{x_1}=0$ in $S/R$. It follows that $(\sum_ik_ic_i)/x_1\in R$ and hence $\sum_{i}k_ic_i\in (x_1)$. This in turn implies that $\overline{c}=0$ and thus,  $\eta$ is injective. It is also surjective as any element $c/x_1$ of $S/R$ has a pre-image $\overline{c}\in (\C_R,x_1)/(x_1)$.
\end{proof}

We shall see that a decrease in the valuation of the conductor can significantly help us in gaining better understanding of torsion elements of $\Omega_R$. If we construct the ring $S=R[\frac{\mathfrak{C}_R}{x_1}]$, then \Cref{conductorbasicprops} guarantees such a drop. We try to explicitly describe the ring $S$ now. First we set up some notation.

\noindent \begin{notation}\label{conventionC/x}
	We know that $\ds \mathfrak{C}_R=(t^c,\dots,t^{c+a_1-1})R$ where $c=c_R$ is the valuation (of the conductor ideal) as discussed in \Cref{conductor def}. Hence $\ds \mathfrak{C}_R/x_1$ is generated in $\overline{R}$ (in fact it is the ideal $(t^{c-a_1})\overline{R}$)
	by the monomials
	\begin{equation}\label{powerlist} t^{c-a_1}, t^{c-a_1+1}, \cdots, t^{c-1}.\end{equation} This is true because if $\alpha_{c-a-1},\ldots,\alpha_{c-1}$ are arbitrary units of the
	integral closure $\overline{R}$, the ideal generated by $t^{c-a_1}, t^{c-a_1+1}, \cdots, t^{c-1}$ is the same as the ideal generated by $\alpha_{c-a-1}t^{c-a_1}, \ldots, \\\alpha_{c-1}t^{c-1}$.
	The new ring $S$ is constructed by adjoining $\mathfrak{C}_R/x_1$. By \Cref{conductorbasicprops}(5) $S/R$ is a $\k$ vector space of dimension $s$. Using \cite[Proposition 2.9]{herzog1971wertehalbgruppe} we see that this is generated by those powers of $t$ from \Cref{powerlist}, which are not in the valuation semigroup of $R$. We call these powers say $b_1,\dots,b_s$ in ascending order. 
\end{notation} 

\noindent\begin{remark}\label{canonicalmoduledescription}  Since $(\C_R,x_1)/(x_1)\subseteq ((x_1):\m_R)/(x_1)$ and the dimension of the latter quantity represents the type of $R$, it is clear that the reduced type $s$ is at most the type of $R$. Since $S$ never equals $R$, it is always at least one. The number $s$ can also be described as $\mu(\omega_R/\omega_S)$ where $\omega_R,\omega_S$ are canonical modules of $R,S$ respectively.
\end{remark}

\begin{proof} For the last statement, dualize the following short exact sequence into the canonical module $\omega_R$ $$0\to R\to S\to \k^s\to 0$$
gives the short exact sequence:
$$0\to \omega_S\to \omega_R\to \Ext^1_R(\k^s, \omega_R)\to 0.$$
Since the number of generators of $\omega_R$ is the type of $R$, and since $\Ext^1_R(\k^s, \omega_R)\cong \k^s$ by duality, the remark follows.
\end{proof}

When $s=1$ we say $R$ is of \textit{reduced type one}. In particular, if $R$ is Gorenstein (type of $R$ equals one), necessarily $s = 1$. The converse is not necessarily true as the next example shows.

\begin{example}
	Let $R=\k\ps{t^4,t^{11},t^{17}}$. We can check $R$ is not Gorenstein using \cite{Kunz1970}. M2 computations show that the conductor $\mathfrak{C}_R=(t^{19})R$. It follows that $S =  \k\ps{t^4,t^{11},t^{15}, t^{16}, t^{17}, t^{18}} = \k\ps{t^4,t^{11},t^{17}, t^{18}}$ and hence $R$ is of reduced type one.
\end{example}

When $\C_R\subseteq \m_R^2$, it follows that $\edim S=n+s$. Recall that canonical ideal $\omega_R$ of $R$ exists (\cite[Proposition 3.3.18 ]{bruns_herzog_1998}).
We can also prove that  $s=\mu_R\left(\frac{\omega_R}{\overline{\m_R\omega_R}\cap\omega_R}\right)$,
where $\overline{\m_R\omega_R}$ is the integral closure of the ideal $\m_R\omega_R$, thinking of the canonical module $\omega_R$ as an ideal of $R$. 
We can show that the canonical module $\omega_S$ of $S$ is in fact $\overline{\m_R\omega_R}\cap\omega_R$.

\begin{theorem}\label{canonical module of S}
	Suppose $S=R[\frac{\C_R}{x_1}]$, then a canonical module $\omega_S$ of $S$ can be chosen to be $\omega_R\cap \overline{\m_R\omega_R}$.
\end{theorem}
\begin{proof}
The first part of the proof is essentially due to \cite[Lemma 3]{Brown_Herzog}. But we provide the proof here with more details suitable for our purposes. Let $(y)$ be a minimal reduction of a canonical ideal $\omega_R$ of $R$. Thus, we have $\omega_R\overline{R}=y\overline{R}$ and $\omega_R\C_R=y\C_R$.  Let  $\omega'_R=\frac{\omega_R}{y}$. Clearly, $R\subseteq\omega'_R\subseteq\overline{R}$. 

Let $Q=\quot(R)$. Now recall that $\omega_R:_Q \omega_R=R$ (\cite[Proposition 3.3.11(c)]{bruns_herzog_1998} and \cite[Lemma 2.4.2]{SwansonHuneke}). It is also well-known that $R:_Q\C_R=\overline{R}$ (see for instance the proof of \cite[Corollary 2.6]{maitra2020partial}). Combining these facts, we obtain that 
\begin{align*}
	\omega_R':_Q\overline{R}=y^{-1}(\omega_R:_Q\overline{R})&=y^{-1}(\omega_R:_Q(R:_Q\C_R))\\
	&=y^{-1}(\omega_R:_Q((\omega_R:_Q\omega_R):_Q\C_R))\\
	&=y^{-1}(\omega_R:_Q(\omega_R:_Q\omega_R\C_R))\\
	&=y^{-1}\omega_R\C_R
\end{align*} where the last equalities follow, by applying duality to the  maximal Cohen-Macaulay module $\omega_R\C_R$ (\cite[Theorem 3.3.10(c)]{bruns_herzog_1998}). The equality $\omega_R\C_R=y\C_R$ now shows that $\omega_R':_Q\overline{R}=y^{-1}\omega_R\C_R=\C_R$.  Since $\omega'_R,\C_R$ are fractional ideals, we have $\Hom(\overline{R},\omega'_R)=\C_R$ \cite[Lemma 2.4.2]{SwansonHuneke}. This implies $\Hom(\C_R,\omega'_R)=\overline{R}$ as $\overline{R}$ is a maximal Cohen-Macaulay module. Thus $\omega'_R:\C_R=\overline{R}$ (again using \cite[Lemma 2.4.2]{SwansonHuneke}).
	
The canonical module $\omega_S\cong\Hom(S,\omega'_R)\cong \omega'_R:S$. Now let $\alpha\in\quot(R)$ such that $\alpha\in \omega'_R:S$. Then $\alpha S\subseteq \omega'_R$ or $\alpha R\left[\frac{\C_R}{x_1}\right]\subseteq \omega'_R$. Thus 
	\begin{align*}
		\alpha R\left[\frac{\C_R}{x_1}\right]\subseteq \omega'_R&\Leftrightarrow \alpha R\subseteq \omega'_R\text{ and }\alpha(\C_R/x)\subseteq \omega'_R\\
		&\Leftrightarrow\alpha R\subseteq \omega'_R\text{ and }\alpha/x\subseteq \omega'_R:\C_R=\overline{R}\\
		&\Leftrightarrow\alpha R\subseteq \omega'_R\text{ and }\alpha\in x\overline{R}=\m_R\overline{R}
	\end{align*}
	Thus $\omega_S\cong \omega'_R\cap \m_R\overline{R} \cong  \omega'_R y\cap \m_R y\overline{R} =\omega_R \cap \overline{\m_R y} =\omega_R \cap \overline{\m_R\omega_R}$.
\end{proof}
Combining the above theorem with \Cref{canonicalmoduledescription}, it is easy to see that $s=\mu_R\left(\frac{\omega_R}{\overline{\m_R\omega_R}\cap\omega_R}\right)$.

\begin{theorem}\label{kernelofnew_proof}
	Let  $\ds \mathfrak{C}_R\subseteq \m_R^2$. Set $P = \k\ps{X_1,...,X_n}$. Choose $S,b_1,\dots,b_s$ as in \Cref{conventionC/x}. Then there exists a presentation of $S$ as follows: $$S=R\Big[\frac{\mathfrak{C}_R}{x_1}\Big]=\frac{\k\ps{X_1,\dots, X_n,T_1,\dots, T_s}}{\ker\Phi + \(X_iT_j-g_{ij}(X_1,\dots,X_n),T_kT_l-h_{kl}(X_1,\dots, X_n)\)_{{1\leq i\leq n, 1\leq j\leq s \atop 1\leq k\leq l\leq s}}}$$ where 
		 $g_{ij}(X_1,\dots,X_n), h_{kl}(X_1,\dots, X_n)\in (X_1,\dots,X_n)^2P$ for all $i,j,k,l$. \\Moreover, $g_{ij}(x_1,\dots,x_n), h_{kl}(x_1,\dots, x_n)\in\C_R$.
	
\end{theorem}

\begin{proof}
	Define  $\Psi : \k\ps{X_1,\dots,X_n,T_1,\dots,T_s}\twoheadrightarrow S$ where $$\Psi(X_i)=\alpha_i t^{a_i},1\leq i\leq n, \qquad\Psi(T_j)=t^{b_j}, 1\leq j\leq s.$$ Note that $c_R-a_1\leq b_1,\dots,b_s\leq c_R-1$.  By same arguments as in  \Cref{conductorbasicprops}$(3)$ and also by reading off valuations, we see that the images of $\ds X_iT_j$ and $\ds T_kT_l$ are all in $\mathfrak{C}_R$. Hence there exists $g_{ij}(X_1,\dots,X_n)$, $h_{kl}(X_1,\dots,X_n)$ such that $$X_iT_j-g_{ij}(X_1,\dots,X_n),T_kT_l-h_{kl}(X_1,\dots, X_n)\in \ker \Psi.$$  Moreover,
	since $\C_R\subseteq \m_R^2$, we can choose $g_{ij}(X_1,\dots,X_n), h_{kl}(X_1,\dots, X_n)\in (X_1,\dots,X_n)^2P$ for all $i,j,k,l$ as well as $g_{ij}(x_1,\dots,x_n), h_{kl}(x_1,\dots, x_n)\in\C_R$. Set $$\ds J=\(X_iT_j-g_{ij}(X_1,\dots,X_n),T_kT_l-h_{kl}(X_1,\dots, X_n)\)_{{1\leq i\leq n, 1\leq j\leq s \atop 1\leq k\leq l\leq s}}.$$ 
	By construction, elements in $J$ do not have any purely linear terms in any $X_i$. By the above discussion, $J\subset \ker \Psi$.
	
	Conversely, let $p(X_1,\dots,X_n,T_1,\dots T_s)\in\ker\Psi$. Modulo the ideal $J$, we can write $$p(X_1,\dots,X_n,T_1,\dots, T_s)\equiv p'(X_1,\dots,X_n) +\sum_{i=1}^s\beta_i T_i$$ where $\beta_i \in \k$ and $p'(X_1,\dots,X_n)\in P$. Since $J\subset\ker\Psi$, we have $p'(X_1,\dots,X_n)+\sum_{i=1}^s\beta_i T_i\in\ker\Psi$. Thus, $\ds \sum_{i=1}^s\beta_i t^{b_i}=\Psi(\sum_{i=1}^s\beta_i T_i)=\Psi(- p'(X_1,\dots,X_n))\in R$. By the choice of $b_i$'s, we immediately obtain that $\sum_{i=1}^s\beta_it^{b_i}=0$. Thus $\beta_i=0$ for all $i$ and hence $p'(X_1,\dots,X_n)\in\ker\Phi$. This shows that $\ker\Psi=\ker\Phi+J$.
	\end{proof}
	
\begin{remark}\label{defining matrix of S}
	Using the defining ideal of the $S$ in the previous theorem gives the following presentation of $\Omega_S$: 
	\begin{equation}
	\scalemath{0.5}{
		\begin{bmatrix}
		~&T_1-\partial g_{11}/\partial x_1&\cdots& \partial g_{n1}/\partial x_1&T_2-\partial g_{12}/\partial x_1&\cdots& \partial g_{n2}/\partial x_1 & \cdots & T_n-\partial g_{1n}/\partial x_1&\cdots& \partial g_{nn}/\partial x_1&\partial h_{11}/\partial x_1&\partial h_{12}/\partial x_1&\cdots &\partial h_{1n}/\partial x_1&\cdots &\partial h_{nn}/\partial x_1\\
		\partial f_i/\partial x_j &\vdots &\vdots &\vdots&\vdots&\vdots&\vdots&\vdots&\vdots&\vdots&\vdots&\vdots&\vdots&\vdots&\vdots&\vdots&\vdots\\
		~&\partial g_{11}/\partial x_n & \cdots &T_1-\partial g_{n1}/\partial x_n&\partial g_{12}/\partial x_n & \cdots &T_2-\partial g_{n2}/\partial x_n & \cdots &\partial g_{1n}/\partial x_n & \cdots &T_n-\partial g_{nn}/\partial x_n&\partial h_{11}/\partial x_1&\partial h_{12}/\partial x_1&\cdots&\partial h_{1n}/\partial x_n&\cdots &\partial h_{nn}/\partial x_n\\
		~& x_1 & \cdots & x_n& 0 & \cdots & 0 &\cdots  & 0 & \cdots & 0 &2T_1 & T_2 &\cdots &0 &\cdots &0\\
		~& 0 & \cdots & 0& x_1 & \cdots & x_n &\cdots  & 0 & \cdots & 0 &0 & T_1 &\cdots &0 &\cdots &0\\
		~& \vdots & \cdots & \vdots& \vdots & \cdots & \vdots &\cdots  & \vdots & \cdots & \vdots &\vdots & \vdots &\cdots &\vdots &\cdots &\vdots\\		
		~& 0 & \cdots & 0& 0 & \cdots & 0 &\cdots  & x_1 & \cdots & x_n &0 & 0 &\cdots &0 &\cdots &2T_n\\
		
		\end{bmatrix}
	}
	\end{equation}
\end{remark}

By abuse of notation, we denote the images of $T_i$ in $S$, by $T_i$ again. Thus $S$ is also the same as $\k\ps{x_1,\dots,x_n,T_1,\dots,T_s}$.  In $S=R\left[\frac{\mathfrak{C}_R}{x_1}\right]=\k\ps{t^{a_1},\dots,\alpha_nt^{a_n},t^{b_1},\dots,t^{b_s} }$, the torsion submodule $\tau(\Omega_S)$ is nonzero (\Cref{two monomials}). 

In $\Omega_S$, notice that $\Psi(T_j)=t^{b_j}$ (refer to \Cref{kernelofnew_proof} for definition of $\Psi$). Clearly, $b_iT_idT_j-b_jT_jdT_i\in \tau(\Omega_S), 1\leq i<j\leq s$ as $b_it^{b_j}dt^{b_i}-b_it^{b_i}dt^{b_j}=0$ in $\Omega_{\overline{R}}$. This torsion element is nonzero due to \cite[Proposition 2.6]{ABC1}. Thus $\Omega_S$ always has nonzero torsion elements of the form $\gamma_{ij}=b_iT_idT_j-b_jT_jdT_i, 1\leq i<j\leq s$ (which are ${s\choose 2}$ in number). So the torsion submodule $\tau(\Omega_S)$ has at least ${s\choose 2}$ elements. Moreover, all these elements are $\k$-linearly independent as follows again from \cite[Proposition 2.6]{ABC1}.

\begin{lemma}\label{one more torsion lemma}

	Let  $\mathfrak{C}_R\subseteq \m_R^2$ and construct $S=R\left[\frac{\mathfrak{C}_R}{x_1}\right]=R[T_1,\dots,T_s]$ as in \Cref{kernelofnew_proof}. Let $\tau(\Omega_R),\tau(\Omega_S)$ represent the torsion submodules of $\Omega_R,\Omega_S$ respectively and $\gamma_{ij}=b_iT_idT_j-b_jT_jdT_i, 1\leq i<j\leq s$. Consider an element $\tau=\sum_i r_idx_i+\sum_j r_j'dT_j$ in $\Omega_S$ where $r_j',1\leq j\leq s$ are not units in $S$. Then there exist $c_{ij}\in \k$ such that $$\tau-\sum_{i,j}c_{ij}\gamma_{ij}=\sum_i r_i''dx_i\in   \Omega_S,$$ where $r_i''\in S$.   In particular, if $\tau\in \tau(\Omega_{S})$ then $\sum_i r_i''dx_i\in   \tau(\Omega_S)$.
\end{lemma}
\begin{proof}
	Since $\m_S^2 = \m_R^2 + \m_R(\frac{\mathfrak{C}_R}{x_1}) + (\frac{\mathfrak{C}_R}{x_1})^2 = \m_R^2$, and since $S/\m_S = R/\m_R$, we can represent the
	entries $r_j', 1\leq j\leq s$ by elements of $R$ plus linear forms over $\k$ in $T_1,...,T_s$. Let $r_j'=p_j(x_1,\dots,x_n)+\sum_{l}k_{jl}T_{l}, 1\leq j\leq s$ where $k_{jl}\in\k$. Clearly, using the Jacobian matrix (as in \Cref{defining matrix of S}) we can rewrite $\tau$ as $\sum_i u_idx_i+\sum_j u_j'dT_j$ where $u_j'=\sum_{l}k'_{jl}T_l$ where $k_{jl}'\in\k$. We wish to eliminate the variables $T_1,\dots,T_s$ from $u_j'$. Suppose $k_{jl}'\neq 0$.
	
	The column in the Jacobian matrix corresponding to the defining equation $T_lT_j-h_{lj}(X_1,\dots,X_n)$ of $S$ (refer to \Cref{kernelofnew_proof}) is of the form $\theta_{lj}=T_ldT_j+T_jdT_l-\sum_t\frac{\partial h_{lj}}{\partial X_t}dX_t$. Notice that $\gamma_{lj}+b_j\theta_{lj}=(b_l+b_j)T_ldT_j-\sum_tb_l\frac{\partial h_{lj}}{\partial X_t}dX_t$. Since $\theta_{lj}=0$ in $\Omega_S$,  $\gamma_{lj}$ can be rewritten as $(b_l+b_j)T_ldT_j-\sum_tb_l\frac{\partial h_{lj}}{\partial x_t}dx_t$ in $\Omega_S$. Now $\tau-\frac{k'_{jl}}{(b_l+b_j)}\gamma_{lj}$ will not have a $T_l$ term in the $(n+j)$-th row (corresponding to $dT_j$). Since $j,l$ were arbitrary, a $\k$-linear combination of $\tau$ and $\gamma_{lj}$ will eliminate all the variables $T_1,\dots,T_s$ from $u_j'$. Since $j$ was arbitrary, we can eliminate the variables $T_1,\dots,T_s$ from $u_1',\dots,u_s'$ as well, to get the result.
\end{proof}

The following theorem is an important technical result of this article. This will serve as the main tool that will help us in pulling back nonzero torsion elements from $\Omega_S$ to $\Omega_{ R}$, as we shall see. We shall write $\ell(M)$ to denote the length of an $R$-module $M$. 

\begin{theorem}\label{thm on one more torsion}
	Let $\mathfrak{C}_R\subseteq \m_R^2$ and construct $S=R\left[\frac{\mathfrak{C}_R}{x_1}\right]=R[T_1,\dots,T_s]$ as in \Cref{kernelofnew_proof}. Let $\tau(\Omega_R),\tau(\Omega_S)$ represent the torsion submodules of $\Omega_R,\Omega_S$ respectively. If $\ell(\tau(\Omega_S))\geq ns+{s \choose 2}+1$ and all these torsion elements have non-units in the last $s$ rows (corresponding to $dT_1,\dots,dT_s$), then a $\k$-linear combination of these torsion elements can be pulled back to a nonzero torsion element in $\tau(\Omega_R)$. In particular, $\tau(\Omega_R)\neq 0$.
\end{theorem}

\begin{remark} Although the hypothesis of at least $ns+{s\choose 2}+1$ $\k$-linear torsion elements sounds very strong, in fact it is not.  We can prove that there are always at least
$ns+{s\choose 2}$  $\k$-linearly independent torsion elements in $\Omega_S$, so this theorem requires only one more new torsion element.  Of course, we are only searching for one
nonzero torsion element in $\Omega_R$ in any case, but the point is that the search for one extra torsion element is often better in $S$ than in $R$, since $S$ is usually
a simpler ring.

The number $ns+\binom{s}{2}$ comes from the following observation: suppose some $\alpha_i=1$. Then we can look at the torsion element $a_ix_idT_j-b_jT_jdx_i$ for $1\leq j\leq s$. We can apply the `monomialization technique' as in the proof of \Cref{valuation more than conductor}, to generate $ns$ such torsion elements; we always have $\binom{s}{2}$ torsion elements coming from the variables $T_j$'s. Finally we  can use \cite[Proposition 2.6]{ABC1} to prove $\k$-linear independence among these. We defer the technical details to a future article.
\end{remark}

\begin{proof}[Proof of \Cref{thm on one more torsion}]
%
%
Since $\ell(\tau(\Omega_S))\geq ns+{s \choose 2}+1$, let $\tau_1,\dots,\tau_{ns+{s\choose 2}+1}$ denote these $\k$-linearly independent torsion elements with non units in the last $s$ rows (corresponding to $dT_1,\dots,dT_s$). By \Cref{one more torsion lemma}, we can use $\gamma_{ij}$ to rewrite $\tau_1,\dots,\tau_{ns+{s\choose 2}+1}$ as $\tau_1',\dots,\tau_{ns+{s\choose 2}+1}'$ which have zeroes in the last $s$ rows. Let $V=\left\{\tau_1',\dots,\tau_{ns+{s\choose 2}+1}'\right\}$ and \\$V'=V\bigcup\left\{\gamma_{ij}~|~1\leq i<j\leq s\right\}$. Thus we have 
\begin{align}\label{one more torsion eq1}
\dim_\k\langle V'\rangle\leq \dim_\k \langle V\rangle +{s\choose 2}	
\end{align}
where $\langle \cdot\rangle$ denotes the $\k$-linear span. Notice that $\{\tau_1,\dots,\tau_{ns+{s\choose 2}+1}\}\subseteq \langle V'\rangle$ and hence
\begin{align}\label{one more torsion eq2}
ns+{s\choose 2}+1\leq \dim_\k \langle V'\rangle.
\end{align}
Combining \eqref{one more torsion eq1},\eqref{one more torsion eq2}, we have $\dim_\k \langle V\rangle\geq ns+1$. Thus there are at least $ns+1$ $\k$-linearly independent elements $\rho_i,1\leq i\leq ns+1$ all having the last $s$ rows (corresponding to $dT_1,\dots,dT_s$) consisting of zeroes.


 Let $B=[\rho_1~\ldots~\rho_{ns+1}]$ be the $(n+s)\times (ns+1)$ matrix obtained by concatenating the column vectors $\rho_i,1\leq i\leq ns+1$.
Since the last $s$ rows of $B$ are zero, it is effectively an $n\times (ns+1)$ matrix.
	
	Our goal is to pullback a $\k-$linear combination of the columns of $B$ to $\Omega_R$. Passing to the vector
	space $S/R$ gives an $n\times (ns+1)$ matrix of linear forms in $T_1,...,T_s$.  Writing the coefficients of each linear form as its own $s\times 1$ column, we obtain an
	$ns\times (ns+1)$ matrix over $\k$.  By elementary column operations over $\k$ it follows that we can obtain a column of zeroes.  Performing the same operations on the matrix $B$ gives
	us a nonzero torsion element whose last $s$ rows are zeroes, and whose entries are in $R$.  These necessarily also represent a torsion element in $\Omega_R$, since they
	are a syzygy of $\frac{dx_1}{dt},\dots,\frac{dx_n}{dt}$.  If this torsion element were zero in $\Omega_R$, it would also be zero in $\Omega_S$, since the presentation of $\Omega_S$ contains the
	Jacobian matrix associated to $R$. \end{proof}

As an immediate application, we generalize a result of Scheja (\cite{scheja1970differentialmoduln}), proved by many researchers, who proved Berger's conjecture in the case $R$ is
quasi-homogeneous. If $R$ is quasi-homogeneous, so too is $S$, so the next result is strictly stronger:

\begin{theorem}\label{quasihomogeneous}
	Let $R=\k\ps{\alpha_1t^{a_1},\dots,\alpha_nt^{a_n}}$ with conductor $\mathfrak{C}_R$. Construct $S=R\left[\frac{\mathfrak{C}_R}{x_1}\right]=R[T_1,\dots,T_s]$ as in \Cref{kernelofnew_proof}. If $S$ is quasi-homogeneous, then $\tau(\Omega_R)$ is nonzero. 
\end{theorem}
\begin{proof} Since $S$ is quasi-homogeneous, 
Scheja in \cite{scheja1970differentialmoduln} showed that $\mu(\tau(\Omega_R))\geq {n+s\choose 2}$ where $\edim S=n+s$.  After possibly a change of generators, we have that \\ $\Omega_S\twoheadrightarrow \n=(x_1,\dots,x_n,T_1,\dots,T_s), dx_i\rightarrow x_i,dT_j\rightarrow T_j$. (This map is well-defined because it is induced from the Euler derivation, see \cite[1.5]{Kunzbook}) Let $\tau$ be a torsion element in $\Omega_S$, and write $\tau=\sum_{i=1}^n r_idx_i+\sum_{j=1}^s r'_jdT_j$. Since $\tau\rightarrow 0$ under the above map, we have $\sum_{i=1}^n r_ix_i+\sum_{j=1}^s r'_jT_j=0$. Thus none of the $r'_j$ can be units in $S$, else it would be a contradiction to $x_1,\dots,x_n,T_1,\dots,T_s$ being a minimal generating set for $\n$. Now since $\Omega_S$ has ${n+s\choose 2}$ $\k$-linearly independent torsion elements and none of the coefficients of $dT_j$ in the description of the torsion elements are units, we can pull back a $\k$-linear combination of these torsion elements to a torsion element in $\Omega_R$ (by \Cref{thm on one more torsion}).
\end{proof}
\begin{example}
	Let $R=\k\ps{ t^5,t^8+t^{11},t^9+t^{11},t^{12}+t^{11}}$. Macaulay2 computations show that the conductor $\mathfrak{C}_R=(t^{13})\overline{R}$ and $S=R\left[\frac{\mathfrak{C}_R}{x_1}\right]$ is the monomial curve $\k\ps{t^5,t^8,t^9,t^{11},t^{12}}$. The embedding dimension of $S$ equals $5$ which is one more than that of $R$. Using the previous theorem we see that the torsion $\tau(\Omega_R)\neq 0$. In fact the deviation $d(R)=\mu(I)-\edim R+1=8-4+1=5$. The defining ideal of $R$ is of height three and is not a Gorenstein ideal.
	\end{example}

\section{Main Results}\label{main results}
In this section we prove some of the main results of this article. In the rest of this section we assume that $\C_R\subseteq \m^2_R$ and  $S=R\left[\frac{\mathfrak{C}_R}{x_1}\right]$ as in \Cref{kernelofnew_proof}. Recall that the $\k$-dimension of $S/R$ is called the reduced type of $R$, denoted by $s$. We proved that $s$ is always at most the type of $R$. In particular, if $R$ is Gorenstein, then $s = 1$. 

\begin{proposition}\label{last row unit}
	Suppose $S=R[\frac{\mathfrak{C}_R}{x_1}]$ and $\tau\in \Omega_S$. If $\tau=\sum r_idx_i+\sum_{j=1}^sr_{n+j}dT_j, r_i\in S$  such that $r_{n+j}$ is a unit in $S$ for some $j$, then for all $1\leq i\leq n$, $x_i\tau\neq 0$.
\end{proposition}
\begin{proof}
	Fix $i\leq n$.  Let $J=\langle x_i^2,x_iT_j,T_j^2, x_1,\dots,x_{i-1}, x_{i+1},\dots,x_n,T_1,\dots,T_{j-1},T_{j+1},\dots,T_s\rangle$. Writing $\overline{(\,)}$ for images in $\Omega_{R/J}$, $\overline{x_i\tau}=\overline{x_ir_i}dx_i+\overline{x_ir_{n+j}}dT_j$. If $r_i$ is not a unit in $S$, then $x_ir_i\in \m_S^2\subseteq J$. Thus $\overline{x_ir_i}dx_i=0$. If $r_i$ is a unit then $\overline{x_ir_i}dx_i=0$ as $x_i^2\in J$. Thus we have $\overline{\tau}=\overline{x_ir_{n+j}}dT_j$ which is nonzero in $\Omega_{S/J}$ by \cite[Proposition 2.6]{ABC1}.
\end{proof}

\begin{corollary}
	Under the hypothesis of the above theorem, for every $\tau=\sum r_idx_i+\sum_{j=1}^sr_{n+j}dT_j\in 0:_{\Omega_S}x_1$,  $r_{n+j}\in S$ cannot be a unit.
\end{corollary}
\begin{proof}
	Follows  from \Cref{last row unit}.
\end{proof}
\begin{lemma}\label{typecomparisonlemma}
	Let $\mathfrak{C}_R\subseteq \m_R^2$ and $S=R[\frac{\mathfrak{C}}{x_1}]$ with $\edim(S)=n+s$ where $s$ is the reduced type of $R$. Then $$\type(S)\leq \type(R)+s(n-1).$$
\end{lemma}

\begin{proof}
	Using \Cref{canonicalmoduledescription}, we get the following short exact sequence
	$$0\to \omega_S\to \omega_R\to \k^s\to 0$$ and tensoring this sequence with $\k$, we get $$\Tor^R_1(\k,\k^s)\to \omega_S\otimes_R \k\to\omega_R\otimes_R \k\to \k^s\to 0.$$ Comparing the length of the modules appearing this short exact sequence now yields, $\dim_\k\Tor^R_1(\k,\k^s)+\type(R)\geq \type(S)+s$ (\cite[Theorem 3.3.11]{bruns_herzog_1998}), and hence $$\type(S)\leq ns-s+\type(R)=\type(R)+s(n-1).$$
\end{proof}

\begin{proposition}\label{powerofm} Let $R$ and $S$ be as in Lemma \ref{typecomparisonlemma}.  Then $\m_S^k\subseteq x_1S$ if and only if $\m_R^k\subseteq (x_1,\C_R)R$.
\end{proposition}

\begin{proof} First we observe that $x_1S = (x_1,\C_R)R$.  Since $\m_R\subseteq \m_S$, one implication follows at once.  To prove the reverse implication, assume that
$\m_R^k\subseteq (x_1,\C_R)R$.  It then suffices to prove that $\m_S^k\subseteq \m_R^k + \C_R$.  A typical term in the expansion of $\m_S^k = (\m_R + \frac{\mathfrak{C}}{x_1})^kS$,
which is not $\m_R^k$, is of the form $\m_R^i(\frac{\mathfrak{C}}{x_1})^{k-i}$ for $0\leq i\leq k-1$. Since $\C_R\subseteq \m_R^2$ and $k-i\geq 1$, this term is contained in
$$\m_R^i(\m_R^{2k-2i-2})(\frac{\C_R}{x_1})\subseteq \C_R$$ as $\m_R\C_R = x_1\C_R$.
\end{proof}

\begin{lemma}[\cite{Guttes1990}]\label{gutteslemma}
	Let $(R,\m_R,\k)$ be a one dimensional complete local reduced $\k$-algebra with $\Char(\k)=0$ and embedding dimension $n\geq 3$. Suppose $y$ is a non-zero divisor such that $\edim(R/R.y)=n-1$
	\begin{itemize}
		\item[$a)$] If $\m_R^4\subseteq R.y$, then \[\ell(0:_{\Omega_R} y)\geq \frac{(n-2)(n-1)}{2}.\]
		
		\item[$b)$] If $\m_R^5\subseteq R.y$, then \[\ell(0:_{\Omega_R} y)\geq \frac{(n-2)(n-1)}{2}-\type(R).\]
	\end{itemize}
\end{lemma}

\begin{proof}
	For $a)$, we refer the reader to the proof of \cite[Satz 4]{Guttes1990}. For $(b)$, see the proof of \cite[Anmerkung, Page 506-507]{Guttes1990}.
\end{proof}

\begin{theorem}\label{m^4 in C}
	Suppose $\m_R^4\subseteq (\C_R,x_1)$ and $n(n - 3)\geq 2s$ where $s$ is the reduced type of $R$, then $\tau(\Omega_R)\neq 0$.
\end{theorem}
\begin{proof}
We first construct $S=R\left[\frac{\C_R}{x_1}\right]$ with $\edim S=n+s$ where $\edim R=n$. By \Cref{powerofm}, we have that $\m^4_S\subseteq x_1S$. Since $x_1$ is a nonzero divisor on $S$, by definition of torsion submodule we have $0:_{\Omega_S}x_1\subseteq \tau(\Omega_S)$. Now using  \Cref{gutteslemma}(a) with $R=S$ and $y=x_1$, we get that 
\begin{align*}
	\ell(\tau(\Omega_S))\geq \ell(0:_{\Omega_S}x_1)\geq\frac{(n+s-2)(n+s-1)}{2},
\end{align*}
where the latter number is more than $ns+{s\choose 2}+1$ when $n(n-3)\geq 2s$. The result now follows from \Cref{thm on one more torsion}.
\end{proof}
\begin{theorem}\label{mainthm}
	Let $S=R\left[\frac{\C}{x_1}\right]$ with $\edim S=n+s$ and assume that $$\displaystyle\text{type}(R)\leq \frac{n^2-3n-2ns}{2}.$$ If $\m_R^5\subseteq \mathfrak{C}_R+ x_1R$, then $\tau(\Omega_R)\neq 0$.
\end{theorem}

\begin{proof}
Since $\mathfrak{C}_R\subseteq \m^2$, 
using  \Cref{powerofm}, we immediately obtain that $\m_S^5\subseteq x_1S$. 
	
Since $x_1$ is a nonzero divisor on $S$, by definition of torsion submodule, $0:_{\Omega_S}x_1\subseteq \tau(\Omega_S)$. By \Cref{gutteslemma}(b) with $R=S$ and $y=x_1$, we get that $$\ell(\tau(\Omega_S))\geq \ell(0:_{\Omega_S}x_1)\geq \frac{(n+s-2)(n+s-1)}{2}-\text{type}(S).$$ 
Here we used that the embedding dimension of $S$ is $n+s$, which follows from the presentation of $S$ given in Remark \ref{defining matrix of S}. All torsion elements in $(0:_{\Omega_S}x_1)$  have non-units in the last row by \Cref{last row unit}. It follows from Theorem \ref{thm on one more torsion} that it suffices to prove 
	 $$\frac{(n+s-2)(n+s-1)}{2}-\text{type}(S)\geq ns+{s\choose 2}+1.$$
	 Using \Cref{typecomparisonlemma}, it suffices to prove that
	$$\frac{(n+s-2)(n+s-1)}{2}-\text{type}(R)-s(n-1)\geq ns+{s\choose 2}+1$$ or 
	$$\text{type}(R)\leq \frac{(n+s-2)(n+s-1)}{2}-2ns+s-{s\choose 2}-1,$$
	which simplifies to our assumption on the type. \end{proof}

\begin{corollary}
	Let $R$ be of reduced type one such that $\text{type}(R)\leq {n\choose 2}-2n$. If $\m_R^5\subseteq \mathfrak{C}_R+ x_1R$, then $\tau(\Omega_R)\neq 0$.
\end{corollary}

\begin{proof} Set $s = 1$ in the above theorem.
\end{proof}

\begin{corollary}\label{m^6 in C}
	Let $R$ be Gorenstein and $n=\edim R\geq 6$. If $\m_R^6\subseteq x_1R$, then $\tau(\Omega_R)\neq 0$.
\end{corollary}
\begin{proof} We may assume that $\mathfrak{C}_R\subseteq \m_R^2$, because we have already shown that the torsion is nonzero in case the inclusion does not hold (\Cref{valuation more than conductor}).
	Since $R$ is Gorenstein, $R$ is also of reduced  type one (\Cref{canonicalmoduledescription}).  The condition that $\m_R^6\subseteq x_1R$ implies that
	$\m_R^5\subseteq (x_1R:\m_R)$.  However, it is always true that $\C_R + x_1R\subseteq x_1:\m_R$, since $x_1\C_R = \m_R\C_R$.  As the conductor can never lie inside a proper principal ideal (follows, for instance, from \cite[Corollary 2.6]{maitra2020partial}) and since $R$ is Gorenstein, we obtain that $\C_R + x_1R = x_1:\m_R$.  Therefore, $\m_R^5\subseteq \mathfrak{C}_R+ x_1R$. The inequality $1=\text{type}(R)\leq {n\choose 2}-2n$ is also satisfied as $n\geq 6$. Thus all the conditions of the previous theorem hold, and the result now follows.
\end{proof}

\begin{remark} As mentioned in the above proof, it is always true that $\C_R + x_1R\subseteq x_1:\m_R$, since $x_1\C_R = \m_R\C_R$.  If equality holds, then $s$ is equal to the type of $R$, and the condition
that $\m_R^5\subseteq \mathfrak{C}_R + x_1R$ is equivalent to the condition that $\m_R^6\subseteq x_1R$.  This case of ``maximal" reduced type gives a further extension of
the work of G\"uttes.  This maximality occurs if $R$ is Gorenstein, but it can also occur in other cases.  For one such example, using Macaulay2, we check that
$R = \k[[t^{10},t^{11}+t^{16},t^{12}+t^{16},t^{13}+t^{16}]]=k[[x,y,z,w$]] has conductor  $\C_R = t^{20}\overline{R}$ and $(x):\m_R=(x,w^2,zw,yw,z^2,yz,y^2)= (x,\C_R)$.  
\end{remark}

\bibliographystyle{siam}
\bibliography{references}

\begin{thebibliography}{10}

\bibitem{Bassein}
{\sc R.~Bassein}, {\em On smoothable curve singularities: local methods}, Math.
  Ann., 230 (1977), pp.~273--277.

\bibitem{Berger63}
{\sc R.~Berger}, {\em Differentialmoduln eindimensionaler lokaler {R}inge},
  Math. Z., 81 (1963), pp.~326--354.

\bibitem{Berger88}
{\sc R.~W. Berger}, {\em On the torsion of the differential module of a curve
  singularity}, Arch. Math. (Basel), 50 (1988), pp.~526--533.

\bibitem{Berger_article}
{\sc R.~W. Berger}, {\em Report on the {T}orsion of {D}ifferential {M}odule of
  an {A}lgebraic {C}urve}, in Algebraic Geometry and its Applications,
  Springer, 1994, pp.~285--303.

\bibitem{Brown_Herzog}
{\sc W.~C. Brown and J.~Herzog}, {\em One-dimensional local rings of maximal
  and almost maximal length}, J. Algebra, 151 (1992), pp.~332--347.

\bibitem{bruns_herzog_1998}
{\sc W.~Bruns and J.~Herzog}, {\em Cohen-Macaulay Rings}, Cambridge Studies in
  Advanced Mathematics, Cambridge University Press, 2~ed., 1998.

\bibitem{Buchweitz}
{\sc R.-O. Buchweitz and G.-M. Greuel}, {\em The {M}ilnor number and
  deformations of complex curve singularities}, Invent. Math., 58 (1980),
  pp.~241--281.

\bibitem{ABC1}
{\sc G.~Corti\~{n}as, S.~C. Geller, and C.~A. Weibel}, {\em The {A}rtinian
  {B}erger conjecture}, Math. Z., 228 (1998), pp.~569--588.

\bibitem{ABC2}
{\sc G.~Cortiñas and F.~Krongold}, {\em Artinian algebras and differential
  forms}, Communications in Algebra, 27 (1999), pp.~1711--1716.

\bibitem{eisenbud_Commalg}
{\sc D.~Eisenbud}, {\em Commutative algebra}, vol.~150 of Graduate Texts in
  Mathematics, Springer-Verlag, New York, 1995.
\newblock With a view toward algebraic geometry.

\bibitem{Guttes1990}
{\sc K.~G{\"u}ttes}, {\em Zum {T}orsionsproblem bei
  {K}urvensingularit{\"a}ten}, Archiv der Mathematik, 54 (1990), pp.~499--510.

\bibitem{Herzog78}
{\sc J.~Herzog}, {\em Ein {C}ohen-{M}acaulay-{K}riterium mit {A}nwendungen auf
  den {K}onormalenmodul und den {D}ifferentialmodul}, Math. Z., 163 (1978),
  pp.~149--162.

\bibitem{herzog1971wertehalbgruppe}
{\sc J.~Herzog and E.~Kunz}, {\em {D}ie {W}ertehalbgruppe eines lokalen {R}ings
  der {D}imension 1}, in Die Wertehalbgruppe eines lokalen Rings der Dimension
  I, Springer, 1971, pp.~3--43.

\bibitem{HerzogWaldi84}
{\sc J.~Herzog and R.~Waldi}, {\em Differentials of linked curve
  singularities}, Archiv der Mathematik, 42 (1984), pp.~335--343.

\bibitem{HerzogWaldi86}
{\sc J.~Herzog and R.~Waldi}, {\em Cotangent functors of curve singularities},
  Manuscripta Math., 55 (1986), pp.~307--341.

\bibitem{Hubl}
{\sc R.~H\"{u}bl}, {\em A note on the torsion of differential forms}, Arch.
  Math. (Basel), 54 (1990), pp.~142--145.

\bibitem{Isogawa}
{\sc S.~Isogawa}, {\em On {B}erger's conjecture about one-dimensional local
  rings}, Arch. Math. (Basel), 57 (1991), pp.~432--437.

\bibitem{Koch}
{\sc J.~Koch}, {\em \"{U}ber die {T}orsion des {D}ifferentialmoduls von
  {K}urvensingularit\"{a}ten}, vol.~5 of Regensburger Mathematische Schriften
  [Regensburg Mathematical Publications], Universit\"{a}t Regensburg,
  Fachbereich Mathematik, Regensburg, 1983.

\bibitem{Kunz1970}
{\sc E.~Kunz}, {\em The value-semigroup of a one-dimensional {G}orenstein
  ring}, Proc. Amer. Math. Soc., 25 (1970), pp.~748--751.

\bibitem{Kunzbook}
\leavevmode\vrule height 2pt depth -1.6pt width 23pt, {\em K\"{a}hler
  differentials}, Advanced Lectures in Mathematics, Friedr. Vieweg \& Sohn,
  Braunschweig, 1986.

\bibitem{maitra2020partial}
{\sc S.~Maitra}, {\em Partial {T}race {I}deals and {B}erger's {C}onjecture},
  arXiv preprint arXiv:2003.11648,  (2020).

\bibitem{pohl1989torsion1}
{\sc T.~Pohl}, {\em Torsion des {D}ifferentialmoduls von
  {K}urvensingularit{\"a}ten mit maximaler {H}ilbertfunktion}, Archiv der
  Mathematik, 52 (1989), pp.~53--60.

\bibitem{Pohl91}
\leavevmode\vrule height 2pt depth -1.6pt width 23pt, {\em Differential modules
  with maximal torsion}, Arch. Math. (Basel), 57 (1991), pp.~438--445.

\bibitem{scheja1970differentialmoduln}
{\sc G.~Scheja}, {\em Differentialmoduln lokaler analytischer {A}lgebren,
  {S}chriftenreihe {M}ath}, Inst. Univ. Fribourg, Univ. Fribourg, Switzerland,
  (1970).

\bibitem{SwansonHuneke}
{\sc I.~Swanson and C.~Huneke}, {\em Integral closure of ideals, rings, and
  modules}, vol.~336 of London Mathematical Society Lecture Note Series,
  Cambridge University Press, Cambridge, 2006.

\bibitem{Ulrich81}
{\sc B.~Ulrich}, {\em {T}orsion des {D}ifferentialmoduls und {K}otangentenmodul
  von {K}urvensingularit\"{a}ten}, Arch. Math. (Basel), 36 (1981),
  pp.~510--523.

\bibitem{yoshino1986torsion}
{\sc Y.~Yoshino}, {\em Torsion of the differential modules and the value
  semigroups of one dimensional local rings}, Mathematics reports, 9 (1986),
  pp.~83--96.

\end{thebibliography}
\end{document}